\documentclass[journal,10pt]{IEEEtran}

\ifCLASSINFOpdf
\else
   \usepackage[dvips]{graphicx}
\fi
\usepackage{url}

\usepackage{cite}
\usepackage{amsmath,amssymb,amsfonts}
\usepackage{graphicx}
\usepackage{textcomp}
\usepackage{xcolor}
\usepackage{multirow}

\usepackage{algpseudocode}
\usepackage{algorithm}

\usepackage{float}
\usepackage{caption}
\usepackage{subcaption}
\usepackage{balance}

\newtheorem{definition}{Definition}

\newtheorem{proof}{Proof}

\newtheorem{theorem}{Theorem}

\begin{document}

\title{$\ell_{1}$-Norm Minimization with Regula Falsi Type Root Finding Methods 
}

\author{Metin Vural, Aleksandr Y. Aravkin, and S\l{}awomir~Sta\'nczak
}


\maketitle

\begin{abstract}
Sparse level-set formulations allow practitioners to find the minimum 1-norm solution subject to likelihood  constraints. Prior art requires this constraint to be convex. In this letter, we develop an efficient approach for nonconvex likelihoods, using Regula Falsi root-finding techniques to solve the level-set formulation.
Regula Falsi methods are simple,  derivative-free, 
and efficient, and the approach provably extends level-set methods to the broader class of nonconvex inverse problems. 
Practical performance is illustrated using $\ell_1$-regularized Student's t inversion, which is a nonconvex approach used to develop outlier-robust formulations. 
\end{abstract}


\IEEEpeerreviewmaketitle

\vspace{-10pt}
\section{Introduction}

\IEEEPARstart{S}{parse} 
recovery using $\ell_{1}$-norm minimization plays a major role in many signal processing applications. 
Denoting $\textbf{y}\in \mathbb{R}^{M}$ as a measurement vector, $\mbox{D} \in \mathbb{R}^{M\times N}$ as an overcomplete matrix with $M<N$, and $\rho$ as the penalty that measures the data misfit, the `noise-aware' level-set problem is to minimize $\ell_{1}$-norm subject to a misfit or likelihood constraint: 
\begin{equation*} 
(\mbox{P}_{\sigma}) \ \ \ \ \	\underset{\textbf{x}\in \mathbb{R}^{N}}{\mbox{minimize}} \ \ \left\|\textbf{x}\right\|_1 \ \ \mbox{s.t.} \ \ \rho(\textbf{y}-\mbox{D}\textbf{x})\leq \sigma, \ \
\label{eq:optimization problem1}
\end{equation*}	
where $\sigma$ indicates the noise level. 
$\mbox{P}_{\sigma}$ is used in many applications, including  compressed sensing \cite{CS1} \cite{CS2}, overcomplete signal representation \cite{OSR1}, \cite{OSR2}, coding theory \cite{Coding1}, and image processing \cite{Image1}. 

An efficient way to solve $(\mbox{P}_{\sigma})$
is to develop an explicit relationship with a 
simpler problem that can be directly solved with primal-only methods, such as the prox-gradient algorithm~\cite{aravkin2019level}:
\begin{equation*}
(\mbox{P}_{\tau}) \ \ \ \ \ \underset{\textbf{x}\in \mathbb{R}^{N}}{\mbox{minimize}} \ \ \ \rho(\textbf{y}-\mbox{D}\textbf{x}) \ \ \mbox{s.t.} \ \ \left\|\textbf{x}\right\|_1\leq\tau. \
\label{eq:optimization problem3}
\end{equation*}
Exploiting the relationship between $\mbox{P}_{\sigma}$ and $\mbox{P}_{\tau}$  allows one to specify the noise tolerance $\sigma$, and then find the solution by inexactly optimizing a sequence of simpler $(\mbox{P}_{\tau})$ problems.

It has been known for a long time that $(\mbox{P}_{\tau})$ and $(\mbox{P}_{\sigma})$ can provide equivalent solutions \cite{Foucart}, and the idea of solving $(\mbox{P}_{\tau})$ to obtain the solution of $(\mbox{P}_{\sigma})$ was first proposed by \cite{Berg,Berg2}. Their idea follows the optimality trade-off between  the minimum $\ell_{1}$-norm and the {least squares} data misfit, which generates a differentiable convex Pareto frontier. This optimality tracing is formulated as a non-linear equation root finding problem, i.e. getting the exact $\tau$ for a given noise tolerance $\sigma$, and is solved by an inexact Newton Method. The resulting {\it level-set} approach has been generalized to other instances of convex programming by~\cite{aravkin2013variational,aravkin2019level}.

In the most general case, the relationship between
 $(\mbox{P}_{\tau})$ and $(\mbox{P}_{\sigma})$ does not require convexity~\cite[Theorem 2.1]{aravkin2013variational}. However, practical implementations of the root-finding approach require convexity of the Pareto frontier to guarantee success of the root finding procedure, limiting the approach to the convex case. Current implementations favor Newton's method, which requires derivatives. To address this issue, an extension using an inexact secant method has also been developed~\cite{aravkin2019level}.
 

In this paper, we introduce \textit{Regula Falsi} type derivative-free non-linear equation root finding schemes to solve $(\mbox{P}_{\sigma})$. They are bracketing type methods that offer convergence guarantee for convex and nonconvex models with the proper choices of root searching interval: two initial points with the opposite signs  assures convergence \cite{Numerical}. \textit{Regula Falsi} type methods do not require convexity to trace the root, allowing nonconvex loss functions in the $(\mbox{P}_{\sigma})$ formulations. Finally, these methods
are also derivative free. All of these advantages allow 
\textit{Regula Falsi} type methods to be applied to cases where Newton, secant, and their variants are not guaranteed to converge. 


Moving outside of the convex class opens the way for using many useful nonconvex models in $(\mbox{P}_{\sigma})$ formulations. For example, in \cite{Stadler2010} and \cite{Buhlmann2011non} consider mixture models whose negative log-likelihood are nonconvex, with applications to
high-dimensional inhomogeneus data where number of covariates could be larger than sample size.  A second application area uses nonconvex Student's t likelihoods to develop outlier-robust approaches \cite{AravkinStudent1,AravkinStudent2, AravkinStudent3}. 
In this paper, we show how \textit{Regula Falsi} type root finding methods can be used with the nonconvex {Student's t} loss, as well 
the convex {least-squares} and {Huber} losses.


This paper is organized as follows. In Section \ref{sec:Pareto}, a Pareto frontier that reveals the relation between $(\mbox{P}_{\tau})$ and $(\mbox{P}_{\sigma})$ is defined and \textit{Regula Falsi} type methods are introduced. Section \ref{sec:Solving Psigma} presents the proposed $(\mbox{P}_{\sigma})$ solver while the simulation results are discussed in Section \ref{sec:Simulations}.

\vspace{-10pt}
\section{Pareto Frontier and Regula Falsi-Type Root Finding Methods}
\label{sec:Pareto}

Under simple `active constraint' conditions, problems $(\mbox{P}_{\tau})$ and $(\mbox{P}_{\sigma})$
are equivalent for some pair $(\tau, \sigma)$~\cite{aravkin2013variational}.
Pareto frontier approaches use root finding and 
inexact solutions to a sequence of 
$(\mbox{P}_{\tau})$  to solve $(\mbox{P}_{\sigma})$.

\vspace{-10pt}
\subsection{Pareto Optimality}

\begin{definition}
	\label{def:pareto}
	i) \textit{Pareto optimal} is the minimal achievable feasible point of a feasible set. 
	ii) The set that comprised of Pareto optimal points is called the \textit{Pareto frontier}.
\end{definition}

In this work, we also seek to solve $(\mbox{P}_{\sigma})$ by working with $(\mbox{P}_{\tau})$. Specifically, we are interested in the
optimal objective value of the $(\mbox{P}_{\tau})$ for a given $\textbf{y}$ and $\tau$ which can be expressed with  following 
\begin{equation}
\nu(\tau):= \underset{
	\textbf{x}
}{\inf}\{ \rho(\mbox{D}\textbf{x} - \textbf{y})  |  \ \left\|\textbf{x}\right\|_1 \leq\tau\},
\label{eq:nu}
\end{equation}
and the corresponding \textit{Pareto frontier} can be defined as
\begin{equation}
\psi(\tau) := \nu(\tau) - \sigma.
\label{eq:psi}
\end{equation}

%
%
%


\begin{theorem}
	\label{theo:1}
	i) 	If $\rho$ is a convex function (e.g. $\ell_{2}$-norm, Huber function), then so is $\psi$.
	ii) If $\rho$ is a nonconvex   function, convexity of $\psi$ does not follow. When $\rho$ is quasi-convex function, then so is $\psi$.
\end{theorem}
\begin{proof}
	Let us consider any two solutions $\textbf{x}_{1}$ and $\textbf{x}_{2}$ of $(\mbox{P}_{\tau})$ for any  $\tau_{1}$ and $\tau_{2}$ respectively. Since $\ell_{1}$-norm is convex, for any $\beta\in [0,1]$ following holds 
	\begin{equation}
	\begin{split}
	\left\|\beta \textbf{x}_{1} +  (1-\beta)\textbf{x}_{2}\right\|_{1} \leq \beta\left\| \textbf{x}_{1}\right\|_{1}  + 
	(1-\beta)\left\| \textbf{x}_{2}\right\|_{1}  \\ = 
	\beta\tau_{1}  + 
	(1-\beta)\tau_{2}. \ \ \ \ \ \ \ \ \
	\end{split}
	\label{eq:conv 1}
	\end{equation}
	An immediate outcome of eq.~\eqref{eq:conv 1} is that  $\beta \textbf{x}_{1} + (1-\beta)\textbf{x}_{2}$ is a feasible point of $(\mbox{P}_{\tau})$ with $\tau = \beta\tau_{1}  + 
	(1-\beta)\tau_{2}$. Thus we can write the following inequality
	\begin{equation}
		\begin{split}
		\nu(\beta\tau_{1}  + 
		(1-\beta)\tau_{2} ) \leq \rho(\mbox{D}(\beta\textbf{x}_{1} + (1-\beta)\textbf{x}_{2}) - \textbf{y}) \ \ \ \ \ \\
		=  \rho(\beta(\mbox{D}\textbf{x}_{1} - \textbf{y})+ (1-\beta)(\mbox{D}\textbf{x}_{2} - \textbf{y}) ).     
		\end{split}
		\label{eq:conv 2}
	\end{equation}
i) If $\rho$ is convex, then
	\begin{equation}
	\begin{split}
 \rho(\beta(\mbox{D}\textbf{x}_{1} - \textbf{y})+ (1-\beta)(\mbox{D}\textbf{x}_{2} - \textbf{y}) )     
	\leq \beta \rho(\mbox{D}\textbf{x}_{1} - \textbf{y}) + \\ (1 - \beta)\rho(\mbox{D}\textbf{x}_{2} - \textbf{y})) 
	= \beta \nu(\tau_{1}) + (1-\beta)\nu(\tau_{2}),  
	\end{split}
\end{equation}
	that shows $\nu$ is convex as well as $\psi$.\\	
	ii) If $\rho$ is quasi-convex, then
	\begin{equation}
	\begin{split}
	\rho(\beta(\mbox{D}\textbf{x}_{1} - \textbf{y})+ (1-\beta)(\mbox{D}\textbf{x}_{2} - \textbf{y}) )     
	\leq \ \ \ \ \ \ \ \ \ \ \ \ \ \ \ \ \ \ \ \\ \max \{\rho(\mbox{D}\textbf{x}_{1} - \textbf{y}),\rho(\mbox{D}\textbf{x}_{2} - \textbf{y})\}	= \max \{\nu(\tau_{1}), \nu(\tau_{2})\}, 
	\end{split}
\end{equation}
that shows $\nu$ is quasi-convex as well as $\psi$.
\end{proof}

Pareto optimal points are unique for $(\mbox{P}_{\tau})$ with convex and quasi-convex losses $\rho$ that can be infered from \cite[Theorem~1.1, Theorem~1.2]{pardalos2017non}, \cite{miettinen2001some}.
Also, the feasible set of $(\mbox{P}_{\tau})$ enlarges as $\tau$ increases, thus $\psi(\tau)$ is nonincreasing \cite{Berg3}. 
In Fig. \ref{fig:pareto 1} an abstract $\psi(\tau)$ is depicted for convex and quasi-convex losses $\rho$ where the red line represents the $\sigma$ level. 

Obtaining the solution of  $(\mbox{P}_{\sigma})$ by solving $(\mbox{P}_{\tau})$ proceeds as follows.  We start with a ${\tau}$ parameter to solve $(\mbox{P}_{\tau})$, and using the solution of $(\mbox{P}_{\tau})$, find a new $\tau$ value. We proceed iteratively until 
$\psi(\tau_\sigma)\to 0$.  
$\tau_\sigma$  occurs at the intersection of the red line and black curve in Fig \ref{fig:pareto 1}, where we immediately see that the solution of  $(\mbox{P}_{\tau})$ is also a solution of the  $(\mbox{P}_{\sigma})$, a fact proven formally by~\cite{aravkin2013variational}.  
Finding $\tau_\sigma$ can be formulated as a nonlinear root finding problem.

\begin{figure} [h]
	\centering
	\begin{subfigure}[b]{0.22\textwidth}
		\includegraphics[width=1\textwidth]{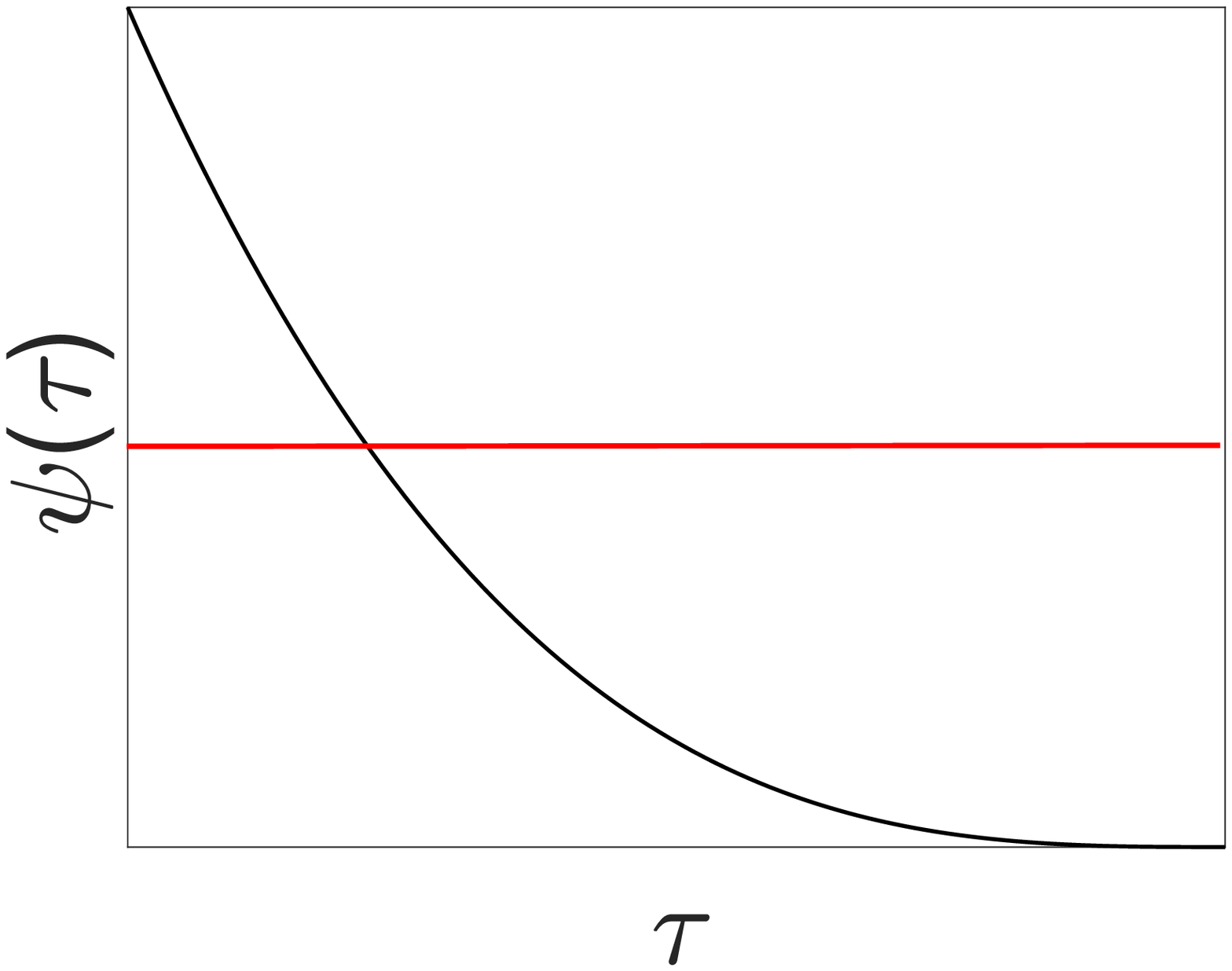}
		\caption{Convex $\rho$}
		\label{fig:a}
	\end{subfigure}
	\begin{subfigure}[b]{0.22\textwidth}
		\includegraphics[width=1\textwidth]{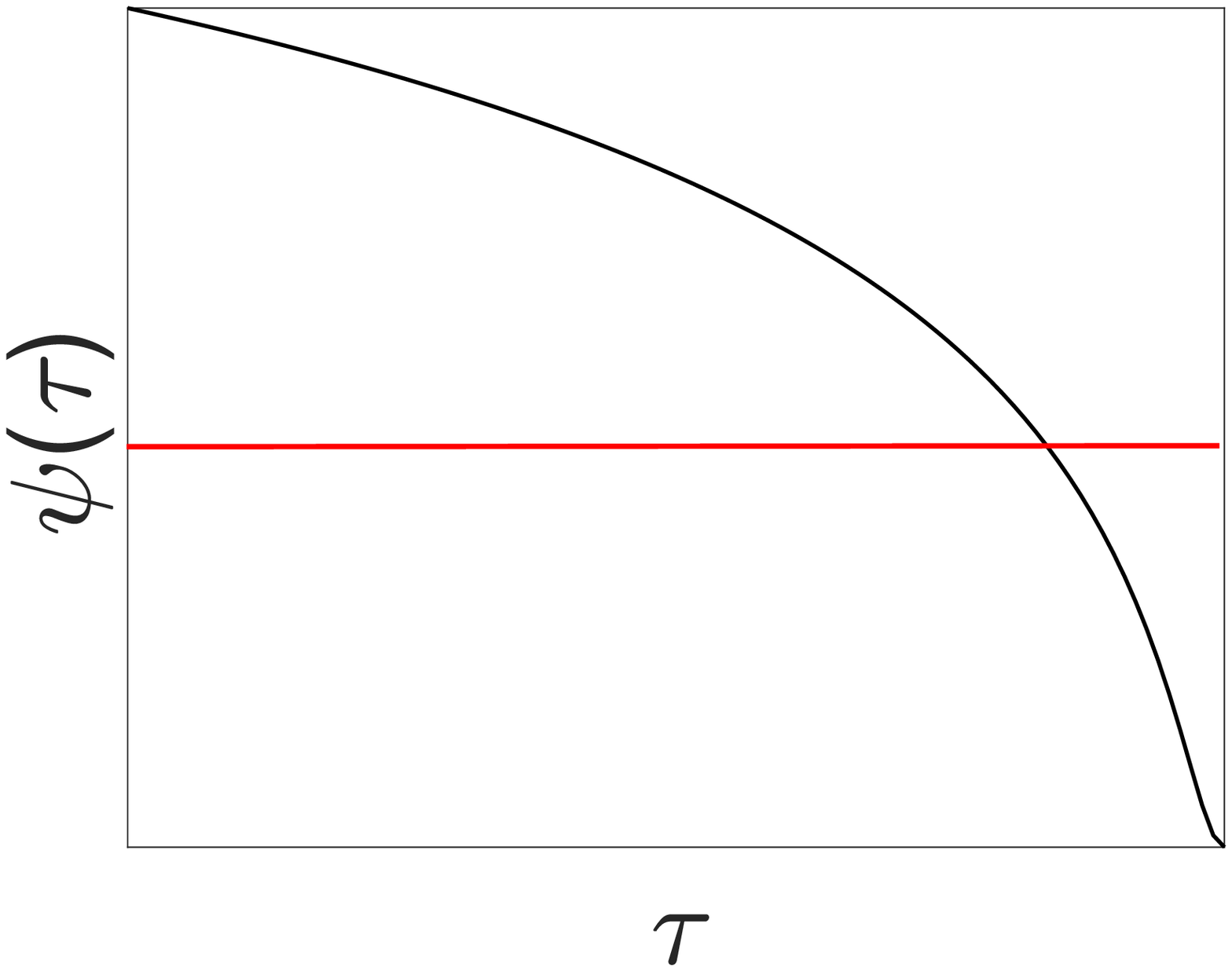}
		\caption{Quasi-convex $\rho$}
		\label{fig:b}
	\end{subfigure} 
	\caption{Pareto frontiers for convex and quasi-convex losses $\rho$.}
	\label{fig:pareto 1}
\end{figure}

\vspace{-10pt}
\subsection{Regula Falsi Type Methods}

Our aim is to 
\begin{equation}
\mbox{find} \ \ \ \tau \ \ \ \mbox{such that} \ \ \ \psi(\tau) = 0.
\label{eq:root 3}
\end{equation}
%
If $\rho$ is nonconvex, neither Newton's method nor secant variants are guaranteed to solve~\eqref{eq:root 3}. 
In particular, the tangent lines may cross in the feasible region, and secant lines may not bracket the feasible area. 
In contrast, regardless of shape of the $\rho$, bracketing type root finding methods are guaranteed to solve~\eqref{eq:root 3}. Here, we develop \textit{Regula Falsi} type methods for~\eqref{eq:root 3}.

We denote the solution of a nonlinear equation of $f$ by $x^{*}$, i.e $f(x^{*}) = 0$. 
With this notation, \textit{Regula Falsi} type methods starting with the points $a$ and $b$ proceed as follows. 

\begin{enumerate}
	\item Calculate the secant line between $a$ and $b$,
	\begin{equation}
	s_{ab} = \frac{f(b) - f(a)}{b - a},
	\label{eq:secant line}
	\end{equation}
	and find the point where (\ref{eq:secant line}) intersects the \textit{x}-axis, which is
$c = b - \frac{f(b)}{s_{ab}}$.
	\item Calculate $f(c)$. If $f(c) = 0$ then $x^{*} = c$, otherwise continue.
	\item Adjust the new interval: if $f(c)f(b)<0$, $x^{*}$ should be in between $b$ and $c$. Set
	\begin{equation}
	a = b, \ b = c, \ \ \text{and} \ \ f(a) = f(b), \ f(b) = f(c),
	\label{eq:assign 1}
	\end{equation}	
	if $f(c)f(b)>0$, $x^{*}$ should be in between $a$ and $c$.  
	\begin{equation}
	b = c, \ \ \text{and} \ \ f(a) =  \mu f(a), \ f(b) = f(c),
	\label{eq:assign 2}
	\end{equation}
	where $\mu$ is the scaling factor.
	\item Check the ending condition: if $|b - a| \leq \epsilon$, stop the iteration. Take 
	\begin{equation}
	x^{*} = 
	\begin{cases}
	b, \ \ \mbox{if} \ \   |f(b)| \leq |f(a) \\
	a,  \ \ \mbox{if} \ \  |f(b)| > |f(a)
	\end{cases},
	\label{eq:assign 3}
	\end{equation}
	if $|b - a| > \epsilon$, continue the iteration, go back to $1)$ with the values $a, b$ and $f(a), f(b)$ from $3)$.
\end{enumerate}

\begin{table}
			\scriptsize
	\caption{\textit{Regula Falsi-type} methods with different $\mu$ values.} 
	\centering
	\begin{tabular}{c | c }
		\hline\hline
		Method & $\mu$ \\ \hline
		Regula Falsi  &  $1$ \\ \hline	
		Illinois  &  $0.5$ \\ \hline	
		Pegasus   &  $\frac{f(b)}{f(b) + f(c)}$ \\ \hline	
		Anderson-Björck   
		& $ 1 - \frac{f(c)}{f(b)}$, and in case $ 1 \leq \frac{f(c)}{f(b)} $  set $\gamma = 0.5$.\\ \hline		
	\end{tabular}
	\label{tab:methods with gamma}
\end{table}

\textit{Regula Falsi} type methods differ from each other in the choice of the scaling factor $\mu$. Several commonly considered $\mu$ in the literature is summarized in Table \ref{tab:methods with gamma}. Additional options for $\mu$ are studied in \cite{Galdino2011family}, \cite{FordImprovedAO}.

\section{Solving $(\mbox{P}_{\sigma})$}
\label{sec:Solving Psigma}

\subsection{$(\mbox{P}_{\tau})$ Solver}

In order to solve $(\mbox{P}_{\sigma})$, we repeatedly solve $(\mbox{P}_{\tau})$. $(\mbox{P}_{\tau})$
can be solved using the simple projected gradient method
\begin{equation}
\textbf{x}^{(k)}  = \text{proj}_{\tau\mathbb{B}_1} \left(\textbf{x}^{(k-1)}+\gamma \mbox{D}^T\nabla \rho(\textbf{y}-\mbox{D}\textbf{x}^{(k-1)})\right).
\label{eq:projection3}
\end{equation}

%
%


\subsubsection{Projection onto the $\ell_{1}$-ball}
\label{sec:proj}
Projection of a vector $\textbf{a} = [ a_1,a_2,...,a_N ]$ onto the $\ell_1$-ball can be written as following 
\begin{equation}
\text{proj} \left(\textbf{a},\tau \right) =  \left\{
\begin{array}{c l}      
\textbf{a}, \ & \text{if} \ \ \left\|\textbf{a}\right\|_1\leq 1 \\
\mbox{sgn}(a_i)\mbox{max}\{ |a_i|-\kappa,0\}, \ & \ \ \ \ \text{else} 
\end{array}\right.
\label{eq:projection l1}
\end{equation}
where $\kappa$ is the Lagrangian multiplier of $\text{proj}_{1}\left(\textbf{a},\tau \right)$ \cite{Duchi}.  
The tricky part of the projection is to find the $\kappa$ that satisifes Karush-Kuhn-Tucker optimality condition $\sum_{i=1}^{N} \left(|a_i|-\kappa\right) = \tau$ in an efficient way. 

To find $\kappa$, we utilized the simple, sorting based approach introduced in \cite{ref:sorting}. Additional variations of this method are described in \cite{Condat}.

\textit{To find $\kappa$:}
\begin{itemize}
	\item Sort $|\textbf{a}|$ as: $c_1\geq c_2 \geq ... \geq c_N$ ,
	\item Find $K = \underset{1\leq k \leq N}{\mbox{max}}\left\{k \ | \left(\sum_{j=1}^{k}c_j - \tau \right)/k\leq c_k\right\}$,
	\item Calculate $\kappa = \left(\sum_{k=1}^{K}c_k - \tau \right)/K$.
\end{itemize}


To solve $(\mbox{P}_{\tau})$, we used a projected gradient method with the spectral line search strategy discussed by  \cite{Berg}, with the projection steps given in \ref{sec:proj}.

\vspace{-15pt}
\subsection{Solving $(\mbox{P}_{\sigma})$}

\subsubsection{Bracketing}
In order to choose the root searching interval, we consider a well-known decomposition 
method called \textit{method of frames (mof)}~\cite{ref:mof}. The \textit{mof} decomposition of a signal $\textbf{y}$ can be obtained with the inverse linear mapping such that
$\textbf{x}_{MF} = \mbox{D}^T(\mbox{D}\mbox{D}^T)^{-1} \textbf{y} = 
\arg \min
\left\{\left\|\textbf{x}\right\|_{2} \ \ \mbox{s.t.} \ \ \textbf{y}=\mbox{D}\textbf{x}\right\}$.
Many common loss functions $\rho$ are nonnegative and vanish at the origin, including gauges and nonconvex losses considered in this study. For these losses, under the assumption that $\mbox{D}$ is full row-rank, $\rho(\textbf{y}-\mbox{D}\textbf{x}_{MF}) = 0$ and $\psi(\tau_{MF})= - \sigma$ with $\tau_{MF} = \left\|\textbf{x}_{MF}\right\|_1$. For the left endpoint, we consider $\textbf{x} = 0$, and $\tau = 0$, with loss equal to $\rho(\textbf{y})$. 
Bracketing the root searching interval between the points $\textbf{x}_{MF}$ and $0$ ensures finding a solution for~\eqref{eq:root 3} since they provide two initial points with opposite signs for $\psi$, as long as $\rho(\textbf{y}) > \sigma$. 


\subsubsection{Solving $(\mbox{P}_{\sigma})$}
We combine the proposed \textit{Regula Falsi} methods and a $(\mbox{P}_{\tau})$ solver to solve  $(\mbox{P}_{\sigma})$ as follows:
\begin{itemize}
	\item Choose initial $\tau$ values. 
	
	Choose two initial values with opposite signs to ensure  convergence of \textit{Regula Falsi} type Methods. The default choice is given by $\tau = 0$ and $\tau = \tau_{MF}$.
	\item Apply the steps of \textit{Regula Falsi} type methods.
	
	Every iteration of the \textit{Regula Falsi-type methods} requires solving $(\mbox{P}_{\tau})$, except at $\tau = 0$ and $\tau = \tau_{MF}$. 
	\item Terminate once the stopping criteria are met.
\end{itemize}

\vspace{-15pt}
\section{Simulations}
\label{sec:Simulations}


In order to examine the performances of \textit{Regula Falsi} type methods given in Table \ref{tab:methods with gamma}, we created a test environment and benchmark by using Sparco framework \cite{sparco}. Real-valued problems are chosen from  Sparco for the simulations among this collection of test problems that includes many examples from the literature.
Details about the problems and related  publications can be found in \cite{sparco}. 

\begin{figure*} [t]
	\centering
	\begin{subfigure}[b]{0.27\textwidth}
		\includegraphics[width=1\textwidth]{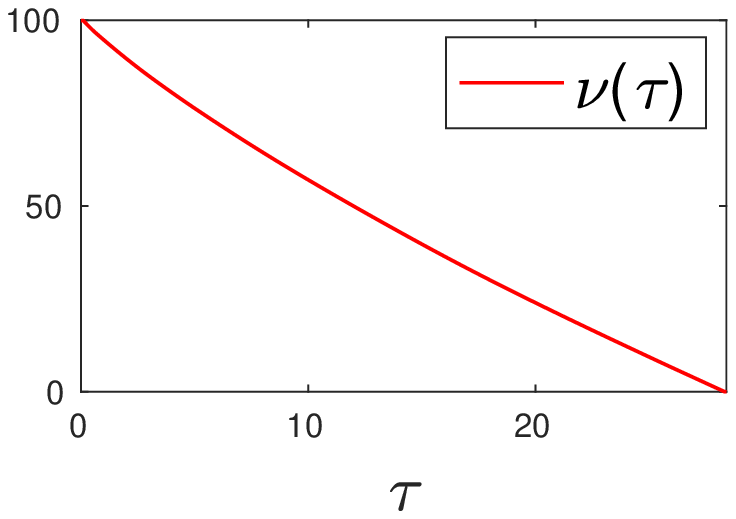}
		\label{fig:aa}
	\end{subfigure}
	\begin{subfigure}[b]{0.27\textwidth}
		\includegraphics[width=1\textwidth]{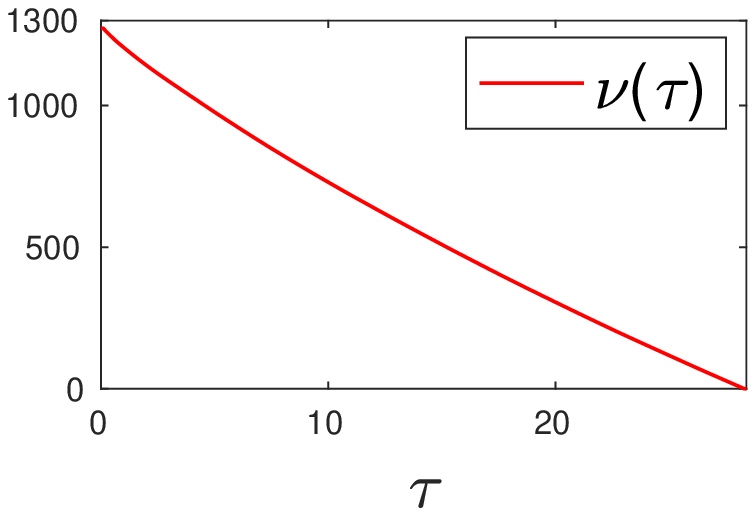}
		\label{fig:bb}
	\end{subfigure} 
		\begin{subfigure}[b]{0.28\textwidth}
		\includegraphics[width=1\textwidth]{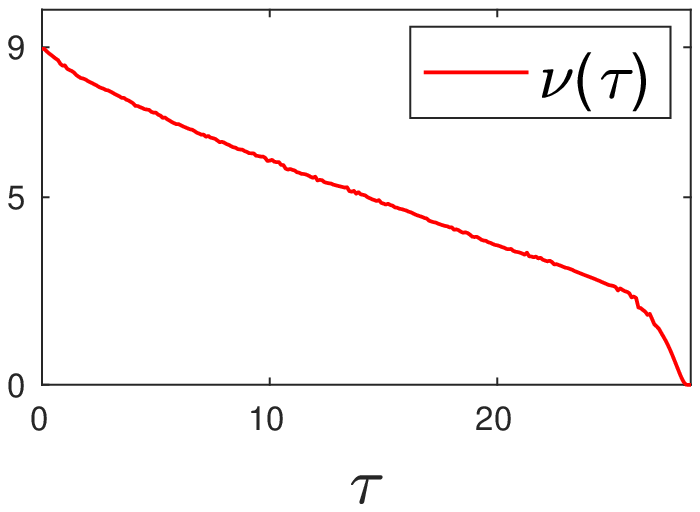}
		\label{fig:cc}
	\end{subfigure} 
	\vspace{-15pt}
	\caption{From left to right, $\nu(\tau)$ for the \textit{gauss-en} problem with $\rho_l$, $\rho_h$ and $\rho_s$.}
	\label{fig:pareto gauss en}
\end{figure*}

Performances of the \textit{Regula Falsi} type methods are investigated for three different loss functions $\rho$ 
which are Least squares $\rho_l(\textbf{x}) = \left\|\textbf{x}\right\|_2$, Huber
\begin{align}
\rho_{h}(\textbf{x})_i = 
\sum_{i=1}^N 
\begin{cases}
\ \ \ \frac{x_i^2}{2 \delta}, \ \ \ \ \ \mbox{if} \ \   |x_i| \leq \delta \\
|x_i|-\frac{\delta}{2}, \ \ \ \mbox{otherwise} 
\end{cases},
\end{align}
and Student's t
$\rho_s(\textbf{x}) = \sum_{i=1}^{N} \nu \log \left(1 + x_i^2/\nu\right)$
where $\delta$ and $\nu$ are the tunning parameters for $\rho_h$ and $\rho_s$, respectively; we take $\delta = 5\times 10^{-3}$ and $\nu = 10^{-2}$ in our simulations.
$N,M$ and $\rho(\textbf{y})$ values for the chosen problems are given in Table \ref{tab:l1 simulations n-m}. 

We solve $(\mbox{P}_{\sigma})$ for a range of $\sigma$ values, chosen relative to $\rho(\mathbf{y})$, in particular  $5\times10^{-1}\rho(\mathbf{y})$,  $5\times10^{-2} \rho(\mathbf{y})$ and $5\times10^{-3} \rho(\mathbf{y})$.
We compute residuals $ \rho(\textbf{r}_{\sigma}) = \rho (\textbf{y}-\mbox{D}\textbf{x}_{\sigma})$, norms $\left\|\textbf{x}_{\sigma}\right\|_1$, the number of nonzero (\textit{nnz}) for each solution  $\textbf{x}_{\sigma}$, and \textit{iter} the total solves of  $(\mbox{P}_{\tau})$ to reach these solutions; results are displayed in Table \ref{tab:l1 simulations}.

Newton's method requires fewer $(\mbox{P}_{\tau})$ solves (see Table \ref{tab:l1 simulations}). However, solving \eqref{eq:root 3} is more expensive for Newton's method than for \textit{Regula Falsi}-type methods since the derivative calculation of the nonlinear equation is required  along with the function evaluation,  while \textit{Regula Falsi}-type methods need only the function evaluation. From a robustness standpoint,  Newton, secant, and their variants are not guaranteed to converge for nonconvex loss functions, in contrast to \textit{Regula Falsi}-type methods. 

The Pareto frontier $\nu(\tau)$ is shown for the \textit{gauss-en} problem with losses $\rho_l$, $\rho_h$ and $\rho_s$ in Figure \ref{fig:pareto gauss en}. We expect similar patterns in many problems. As expected, the Pareto frontier is nonconvex, and  Newton iterations will not stay in the feasible area for $\rho_s$  in this example, since the tangent lines can leave the feasible region.

Figure \ref{fig:cs} shows a typical compressed sensing example. A $20$-sparse vector is recovered using a normally distributed Parseval frame $\mbox{D} \in \mathbb{R}^{175\times 600}$. A measurement is generated according to 
	$\textbf{y} = \mbox{D}\textbf{x} + \textbf{w} + \zeta$,
where the noise $\textbf{w}$ is zero mean normal error with the variance of $0.005$ and $\zeta$ has five randomly placed outliers with a zero mean normal distribution variance of $4$. $(\mbox{P}_{\sigma})$ solved with the $\sigma = \rho(\zeta)$ for a fair comparison of $\rho_l$, $\rho_h$ and $\rho_s$. Huber loss is less sensitive to outliers  in measurement data than least-squares, but Student's t outperforms Huber loss since it grows sublinearly as outliers increase, a property also noted by \cite{aravkin2013variational}.

\begin{table} [t]
		\scriptsize
	\centering
	\begin{tabular}{l c c c c c c }
		\hline\hline
		Problems & id & $M$ & $N$ & $\rho_l(\textbf{y})$ & $\rho_h(\textbf{y})$ & $\rho_s(\textbf{y})$ \\ \hline	
		cos-spike & $3$ & $1024$ & $2048$ & $102.2423$ & $2378.8$ & $25.09$ \\ 	
gauss-en & $11$ & $256$ & $1024$ & $99.9055$ & $1272.5$ & $8.957$\\ 
		jitter & $902$ & $200$ & $1000$ & $0.4476$ & $4.6881$ & $0.0901$\\ 		
		\hline
	\end{tabular}
	\caption{$N,M$, $\rho(\textbf{y})$ values for the problem setups.} 
	\label{tab:l1 simulations n-m}
\end{table}

\vspace{-10pt}
\section{Conclusion}
In this note, we developed a new approach using bracketing \textit{Regula Falsi}-type methods, that enable level-set methods to be applied to sparse optimization problems with nonconvex likelihood constraints, significantly expanding their usability. These methods achieve comparable performance to Newton's method for root finding on convex problems, and are guaranteed to converge in the nonconvex case, where Newton and secant variants may fail.  

%
%


\begin{figure*} 
	\centering
	\begin{subfigure}[b]{0.4925\textwidth}
		\includegraphics[width=9.2cm,height=4.5cm]{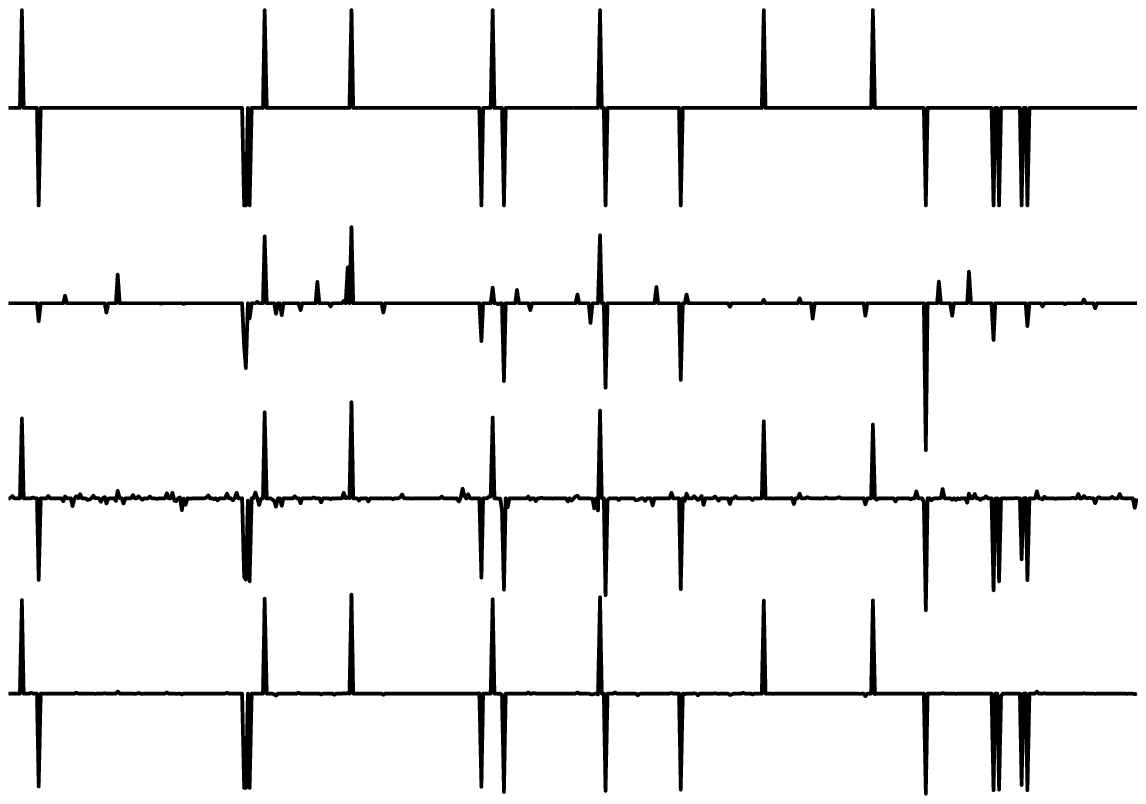}
		\label{fig:aa}
	\end{subfigure}
\begin{subfigure}[b]{0.4925\textwidth}
		\includegraphics
		[width=9.2cm,height=4.5cm]{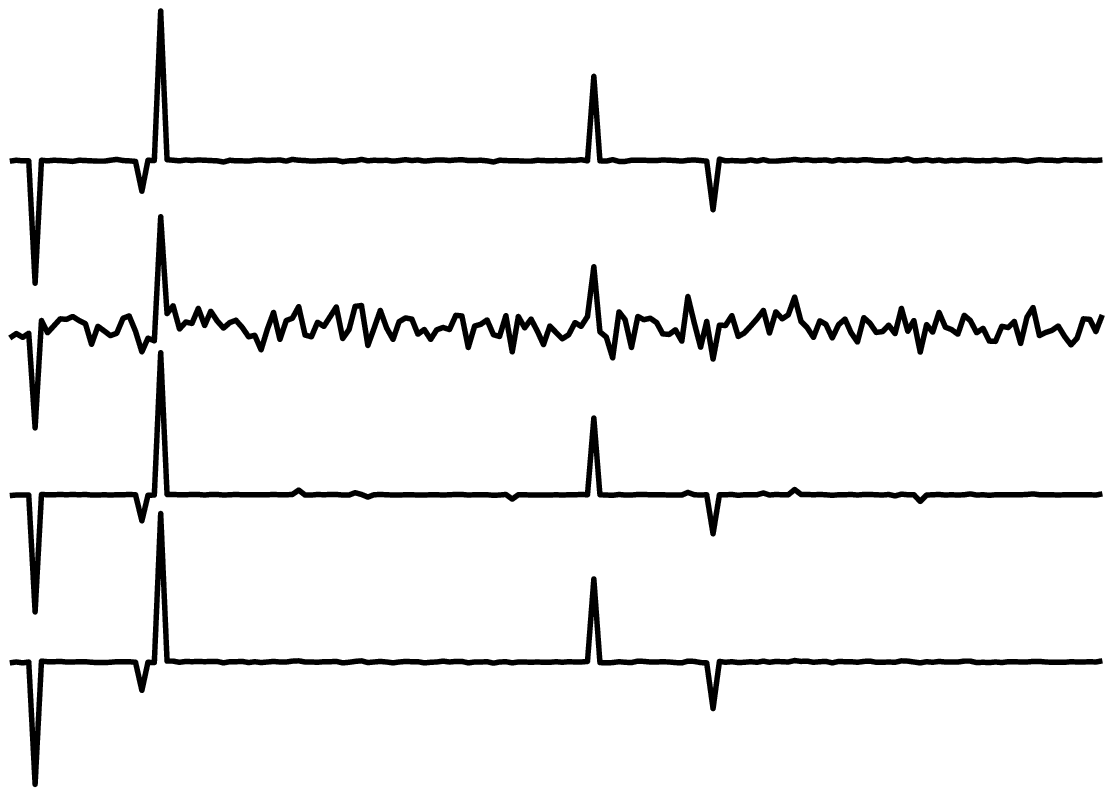}
		\label{fig:bb}
	\end{subfigure} 
	\vspace{-25pt}
	\caption{Left, top to bottom: true signal, reconstructions with \textit{least squares}, \textit{Huber} and \textit{Student's t} losses.
	Right, top to bottom: true errors, \textit{least squares}, \textit{Huber} and \textit{Student's t} residuals.
}
	\label{fig:cs}
\end{figure*}

\vspace{-80pt}

\begin{table*} 
		\scriptsize
	\centering
	\caption{Simulation Results for Solving $(\mbox{P}_{\sigma})$.} 
	\begin{tabular}{l c c | c c c c| c c c c | c c c c | }
		\hline\hline
		& & & \multicolumn{4}{c|}{\textit{least squares}}  &\multicolumn{4}{c|}{\textit{Huber ($\delta = 5\times 10^{-3}$)}} &   \multicolumn{4}{c|}{\textit{Student's t ($\nu = 10^{-2}$)}}  \\ \cline{4-7}  \cline{8-11} \cline{12-15}
		Problems & $\sigma / \rho(\mathbf{y})$ & Methods & $\rho_l(\textbf{r}_{\sigma})$ &$\left\|\textbf{x}_{\sigma}\right\|_1$ & $nnz$ & $iter$ &  $\rho_h(\textbf{r}_{\sigma})$ &$\left\|\textbf{x}_{\sigma}\right\|_1$ & $nnz$ & $iter$  & $\rho_s(\textbf{r}_{\sigma})$ &$\left\|\textbf{x}_{\sigma}\right\|_1$ & $nnz$ & $iter$ \\ \hline 

		\multicolumn{1}{l}{\multirow{15}{*}{cos-spike}} & \multicolumn{1}{c}{\multirow{5}{*}{0.5}} & Regula Falsi &$51.12$ & $66.24$ & $2$ & $10$ & $1189.4$  & $68.328$ & $2$ & $8$ & $12.545$ & $94.289$ & $10$ & $14$  \\	
		& & Illinois & $51.12$ & $66.24$ & $2$ & $7$ & $1189.4$ & $68.328$ & $2$ & $7$ & $12.545$ & $94.289$ & $10$ & $7$ \\
		& & Pegasus & $51.12$ & $66.24$ & $2$ & $7$ & $1189.4$ & $68.328$ & $3$ & $6$ & $12.545$ & $94.289$ & $10$ & $6$ \\
		& & And.-Björck  & $51.12$ & $66.24$ & $2$  & $9$ & $1189.4$ & $68.328$ & $3$ & $9$ & $12.545$ & $94.289$ & $10$ & $10$ \\
		& & Newton  & $51.12$ & $66.24$ & $2$ & $3$ & $1189.4$ & $68.328$ & $2$  & $3$ & $-$ & $-$ & $-$ &  $-$ \\ \cline{2-15}
		& \multicolumn{1}{c}{\multirow{5}{*}{0.05}} & Regula Falsi & $5.112$ & $188.4$ & $75$ & $85$ & $118.94$ & $134.71$ & $2$ & $68$ & $1.2105$ & $171.09$ & $79$ & $72$ \\		
		& & Illinois & $5.112$ & $188.4$ & $75$ & $13$ & $118.94$ & $134.71$ & $2$ & $15$ & $1.2214$ & $172.65$ & $78$ & $16$ \\
		& & Pegasus & $5.112$ & $188.4$ & $75$ & $12$ & $118.94$ & $134.71$ & $2$ & $14$ & $1.259$ & $170.95$ & $76$ & $34$ \\
		& & And.-Björck  & $5.112$ & $188.4$ & $75$ & $27$ & $118.94$ &$134.71$  & $2$ & $32$ & $1.251$ & $170.85$ & $75$ & $50$ \\
		& & Newton  & $5.112$ & $188.4$ & $75$  & $5$ & $118.94$ & $134.7$ & $2$   & $4$ & $-$ &$-$ &  $-$& $-$\\ \cline{2-15}
		& \multicolumn{1}{c}{\multirow{5}{*}{0.005}} & Regula Falsi & $0.511$ & $235.6$ & $123$ & $396$ & $11.893$ & $229.02$ & $127$ & $451$ & $0.1272$ & $231.15$ & $124$ & $383$ \\	
		& & Illinois & $0.511$ & $235.6$ & $123$ & $19$ & $11.894$  & $229.02$ & $125$ & $15$ & $0.1183$ & $231.15$ & $124$ & $30$ \\
		& & Pegasus & $0.511$ & $235.6$ & $123$ & $19$ &   $11.894$ & $229.02$ & $130$ & $14$ & $0.118$ & $231.16$ & $121$ & $27$ \\
		& & And.-Björck  & $0.511$ & $235.6$ & $123$ & $34$ & $11.894$ & $229.02$ & $130$ & $21$ & $0.1273$ & $231.13$ & $122$ & $19$ \\ 
		& & Newton  & $0.511$ & $235.6$ & $123$ & $5$ & $11.894$ & $229.02$ & $125$ & $3$ & $-$ &$-$ &$-$  &$-$ \\ \hline 

		\multicolumn{1}{l}{\multirow{15}{*}{gauss-en}} & \multicolumn{1}{c}{\multirow{5}{*}{0.5}} & Regula Falsi & $49.95$ & $11.9$ & $16$ & $9$ & $636.27$  & $11.97$ & $19$ & $9$ & $4.479$ & $16.6$ & $139$ & $11$  \\	
& & Illinois & $49.95$  & $11.9$ & $16$ & $7$ & $636.27$ & $11.97$  & $19$ & $6$ & $4.479$ & $16.6$ & $139$ & $8$ \\
& & Pegasus & $49.95$  & $11.9$ & $16$ & $5$ & $636.27$ & $11.97$ & $19$  & $5$ & $4.479$ & $16.6$ & $139$ & $6$ \\
& & And.-Björck  & $49.95$  & $11.9$ & $16$  & $7$ & $636.27$ & $11.97$  & $19$ & $7$ & $4.479$ & $16.6$ & $139$ & $8$ \\
		& & Newton  & $49.95$ & $11.9$ &  $16$  & $4$  & $636.27$ & $11.97$ & $19$ & $4$ & $-$ & $-$ & $-$ & $-$ \\ \cline{2-15}
& \multicolumn{1}{c}{\multirow{5}{*}{0.05}} & Regula Falsi & $4.995$ & $26.45$ & $35$ & $26$ & $63.63$  & $26.406$ & $50$ & $27$ & $0.448$ & $27.82$ & $39$ & $24$ \\		
& & Illinois & $4.995$ & $26.45$ & $35$ & $11$ &  $63.63$ & $26.406$ & $50$ & $14$ & $0.448$ & $27.82$ & $39$ & $15$ \\
& & Pegasus & $4.995$ & $26.45$ & $35$ & $11$ & $63.63$ &  $26.406$ & $50$ & $11$ & $0.448$ & $27.82$ & $39$ & $15$ \\
& & And.-Björck  & $4.995$ & $26.45$ & $35$ & $15$ & $63.63$ & $26.406$ & $50$ & $27$ & $0.448$ & $27.82$ & $39$ & $28$ \\
		& & Newton  & $4.995$  & $26.45$ &  $35$  & $4$ & $63.63$ & $26.406$ & $50$ & $4$ & $-$&$-$ & $-$ &$-$ \\ \cline{2-15}
& \multicolumn{1}{c}{\multirow{5}{*}{0.005}} & Regula Falsi & $0.499$ & $28.06$ & $35$ & $208$ & $6.3624$ & $28.036$ & $46$ & $209$ & $0.045$ & $28.12$ & $1024$ & $207$ \\	
& & Illinois & $0.499$ & $28.06$ & $35$ & $19$ & $6.3627$ & $28.036$ & $46$ & $19$ & $0.045$ & $28.12$ & $1024$ & $27$ \\
& & Pegasus & $0.499$ & $28.06$ & $35$ & $19$ & $6.3627$ & $28.036$ & $46$ & $20$ & $0.045$ & $28.12$ & $1024$ & $24$ \\
& & And.-Björck  & $0.499$ & $28.06$ & $35$ & $58$ & $6.3627$ & $28.036$ & $46$ & $62$ & $0.045$ & $28.12$ & $1024$ & $65$ \\ 
		& & Newton  & $0.499$ & $28.06$ &  $35$ & $3$ & $6.3628$ & $28.036$ & $46$ & $4$ & $-$&$-$ & $-$ &$-$ \\ \hline

		\multicolumn{1}{l}{\multirow{15}{*}{jitter}} & \multicolumn{1}{c}{\multirow{5}{*}{0.5}} & Regula Falsi & $0.224$ & $0.766$ & $3$ & $7$ & $2.344$ & $0.6959$ & $3$ & $7$ & $0.045$ & $0.4548$  & $2$ & $9$  \\	
& & Illinois & $0.224$ & $0.766$  & $3$ & $6$ & $2.344$ & $0.6958$ & $3$ & $7$ & $0.045$ & $0.4548$  & $2$ & $7$ \\
& & Pegasus & $0.224$ & $0.766$  & $3$ & $6$ & $2.344$ & $0.6958$ & $3$ & $6$ & $0.045$ & $0.4548$ & $2$ &  $6$ \\
& & And.-Björck  & $0.224$ & $0.766$  & $3$  & $8$ & $2.344$ & $0.6958$ & $3$ & $8$ & $0.045$ & $0.4548$ & $2$ & $9$ \\
		& & Newton  & $0.224$ & $0.766$ & $3$  & $3$ & $2.344$ & $0.6958$ & $3$ & $3$ & $-$ & $-$& $-$ &$-$ \\ \cline{2-15}
& \multicolumn{1}{c}{\multirow{5}{*}{0.05}} & Regula Falsi & $0.022$ & $1.537$ & $3$ & $39$ & $0.2343$ & $1.4338$ & $3$ & $43$ & $0.005$ & $1.2587$  & $3$ & $54$ \\		
& & Illinois & $0.022$  & $1.537$ & $3$ & $12$ & $0.2344$ & $1.4338$ & $3$ & $15$ & $0.005$ & $1.2585$ & $3$ & $14$ \\
& & Pegasus & $0.022$  & $1.537$ & $3$ & $12$ & $0.2344$ & $1.4338$ & $3$ & $14$ & $0.005$ & $1.2585$ & $3$ & $12$ \\
& & And.-Björck  & $0.022$  & $1.537$ & $3$ & $30$ & $0.2344$ & $1.4338$ & $3$ & $39$ & $0.005$ & $1.2585$ & $3$ & $41$ \\
		& & Newton  & $0.022$ & $1.537$ & $3$  & $2$  & $0.2344$ & $1.4338$ & $3$ & $4$ & $-$& $-$& $-$ & $-$\\ \cline{2-15}
& \multicolumn{1}{c}{\multirow{5}{*}{0.005}} & Regula Falsi & $0.002$ & $1.614$ & $3$ & $375$ & $0.0233$ & $1.5646$ & $3$ & $391$ & $0.0005$ & $1.5086$ & $3$ & $413$ \\	
& & Illinois & $0.002$ & $1.614$ & $3$ & $19$ & $0.0234$ & $1.5644$ & $3$ & $16$ & $0.0005$  & $1.508$ & $3$ & $16$ \\
& & Pegasus & $0.002$ & $1.614$ & $3$ & $20$ & $0.0234$ & $1.5644$ & $3$ & $14$  & $0.0005$  & $1.508$ & $3$ & $16$ \\
& & And.-Björck  & $0.002$ & $1.614$ & $3$ & $35$ & $0.0234$ & $1.5644$ & $3$ & $27$ & $0.0005$  & $1.508$ & $3$ & $107$ \\ 
		& & Newton  & $0.002$ & $1.614$ & $3$  & $2$  & $0.0235$ & $1.5643$ & $3$ & $5$ & $-$ & $-$&$-$  &$-$ \\\hline

	\end{tabular}
	\label{tab:l1 simulations}
\end{table*}

\clearpage
\balance
\bibliographystyle{IEEEtran}
\bibliography{IEEEabrv,references}

%
%

\end{document}